\documentclass[11pt]{article}
\usepackage[a4paper]{geometry}

\usepackage{amsmath}
\usepackage{amssymb}
\usepackage{amsthm}
\usepackage{mathrsfs}
\usepackage{bbm}
\usepackage{empheq}
\usepackage[loose]{subfigure}
\usepackage{epsfig}
\usepackage{graphicx}
\usepackage{psfrag}
\usepackage[usenames,dvipsnames]{pstricks}
\usepackage{pst-plot}
\usepackage[colorlinks]{hyperref}
\usepackage{tabls}
\usepackage{paralist}
\usepackage{hyperref}

\DeclareSymbolFontAlphabet{\Bbb}{AMSb}

\newlength{\fixboxwidth}
\newcommand{\COMMENT}[1]{}

\DeclareMathOperator{\ext}{ext}

\newcommand{\eins}{\mathbbm{1}}

\newcommand{\R}{\mathbb{R}}

\newcommand{\quark}{\setbox0\hbox{$x$}\hbox to\wd0{\hss$\cdot$\hss}}

\newcommand{\snorm}[1] {\Vert #1 \Vert}

\newtheorem{thm}{Theorem}[section]

\newtheorem{lem}[thm]{Lemma}

\newtheorem{cor}[thm]{Corollary}
\theoremstyle{definition}

\newtheorem{rmk}[thm]{Remark}

\def \d         { \delta }
\hypersetup{
    linkcolor=blue,
}

\title{Extreme points of a ball about a measure with finite support
}
\author{Houman Owhadi and Clint Scovel
\\
California Institute of Technology
}
\date{\today}

\makeatletter
\@addtoreset{figure}{section}
\@addtoreset{table}{section}
\makeatother

\renewcommand{\theequation}{\arabic{section}.\arabic{equation}}
\renewcommand{\thefigure}{\arabic{section}.\arabic{figure}}
\renewcommand{\thetable}{\arabic{section}.\arabic{table}}

\makeatletter
\renewcommand{\p@subfigure}{\thefigure}
\makeatother

\newcounter{mycount}

\begin{document}
\maketitle
          \begin{abstract}
We show that, for the space of Borel probability measures on a Borel subset of a Polish metric space, the extreme points of the Prokhorov, Monge-Wasserstein and Kantorovich metric balls about a measure whose support has at most $n$ points, consist of measures whose supports have at most $n+2$ points.
 Moreover, we use  the
 Strassen and Kantorovich-Rubinstein duality theorems  to  develop
 representations of
supersets of the extreme points based on linear programming, and then develop these representations towards
the  goal of their efficient computation. 
          \end{abstract}

          \section{Introduction}\label{intro}
In a recent work by  Wozabal \cite{Wozabal}, a framework for optimization under ambiguity is developed
 -including
a discussion of the history of the subject and the current literature. See also 
 Dupa{\v{c}}ov{\'a} \cite{Dupavcova2011} and the recent work by
Esfahani and Kuhn \cite{esfahani2015data}, which expands Wozabal's approach to develop
convex reductions for an important class of objective functions.
We quote from the abstract:
``Though the true distribution is unknown, existence
of a reference measure $P$ enables the construction of non-parametric ambiguity sets as
Kantorovich balls around $P$. The original stochastic optimization problems are robustified
by a worst case approach with respect to these ambiguity sets.''
Fundamental to the development of this framework, Wozabal
  \cite[Cor.~1]{Wozabal}  asserts that, when the domain is a compact metric space,
   the
  extreme points of a Kantorovich ball about a measure whose support has at most $n$ points consist
of measures whose supports have at most $n+3$ points.
The purpose of this paper is to extend and sharpen this result;
extending the domain from
 a compact metric space to  a Borel subset of a Polish metric space,
and improving the bound on the number of Dirac masses from $n+3$ to $n+2$. In addition, we provide  similar results for the Prokhorov metric  and for the Monge-Wasserstein distances.
This increase in generality from a compact metric space to a Borel subset of a Polish space has two nontrivial components. The first is that it replaces compactness with separability. That is, since a compact metric space is complete, it amounts to a generalization from compact complete metric spaces to separable complete metric spaces. The second is that it replaces completeness with measurability. That is, 
it eliminates the completeness requirement and substitutes it with the requirement that it be a Borel subset of separable complete metric space. For example, these results now apply to the case of probability measures on the (noncompact) open interval (0,1).

To outline how they are obtained, recall
Rogosinski's Lemma \cite{Rogosinski}, that   on an arbitrary measurable space, the $n$
moments corresponding to the expected values of $n$ integrable functions with respect to a
 probability measure can be achieved by
a convex sum of $n+1$ Dirac masses. Moreover, recall that an exposed point of a convex set in a locally convex space
 is a point
which is the unique maximizer of some continuous affine function, and
Straszewicz \cite{Straszewicz} Theorem, that the exposed points
of a finite dimensional  compact convex set  is dense in its extreme points.
Wozabal uses the  Kantorovich-Rubinstein Theorem combined with Rogosinski's Lemma \cite{Rogosinski} to characterize
the exposed points of the Kantorovich ball about a measure whose support has at most $n$ points to be
a measure with support at most $n+3$ points. The fact that one obtains $n+3$ Dirac masses comes from
the fact that Kantorovich-Rubinstein theorem introduces one function, the notion  of an exposed point another,
and the central measure having support of size $n$ introduces $n$ more functions, leading to a total of
$n+2$ continuous functions on the set of probability measures on $X\times X$, so that
  Rogosinski's Lemma   implies that
the exposed points
are convex sums of $(n+2) +1 =n+3$ Dirac masses. Then,
 Choquet's \cite[Sec.~17, pg.~99]{Choquet1969lectures}
extension of Straszewicz' Theorem \cite{Straszewicz} to compact metrizable subsets of locally convex space along with
the fact that the set of probability measures equipped with the weak topology is compact and metrizable when the domain is,
is used to show that these  exposed points are dense in the extreme points. A limiting argument showing
that the weak limit of a convex sum of $n+3$ Dirac masses is a convex sum of $n+3$ Dirac masses
establishes the assertion.

In our approach, we  use Dudley's \cite[Thm.~11.8.2]{Dudley:2002}  version of
the Kantorovich-Rubinstein Theorem
 for tight measures on separable metric spaces, and
characterize the extreme points of the space of measures corresponding to the Kantorovich-Rubinstein duality
 using
 results of Winkler \cite{Winkler:1988,winkler1978integral}, previously applied
in \cite{OSSMO:2011} to the  reduction of optimization problems on {\em non-compact} spaces of {\em tight} probability measures arising in Uncertainty Quantification.
Since, by Suslin's Theorem,  a Borel subset of a Polish space is Suslin and since all probability measures
on Suslin spaces are tight, these results allow the extension of many results regarding the extreme points
of sets of probability measures from compact metric domains and continuous moment functions
 to Borel subsets of Polish metric spaces and measurable moment functions.
Then a fundamental result
 that is implicit in the results of Winkler \cite{Winkler:1988,winkler1978integral} is proven in Theorem \ref{thm_winkler};  that
a weakly closed convex set of probability measures  on a Borel subset of a Polish metric space has an extreme point.
This result combined with Lemma \ref{lem_affineedtreme2},  giving sufficient conditions
that the affine image
of the extreme points of a set cover  the extreme points of the affine image of that set,
 shows that the image of these extreme points in the dual cover the extreme points
of the Kantorovich ball. This latter approach has the advantage that it does not pass through the intermediate
stage of exposed points, so does not add an additional function, and does
not require a generalization of Straszewicz' Theorem \cite{Straszewicz} to non-compact sets, although it does suggest
that such a generalization may exist for weakly closed convex sets of tight measures.

To establish our main result, Theorem \ref{thm_wozabal}, we develop a more general and expressive result
in Theorem \ref{thm_wasser}, which not only produces a similar result for the 
Monge-Wasserstein
metric, its   
  Corollary \ref{cor_wasser} shows  how the  duality results of Kantorovich-Rubinstein and Strassen
combined with the results of Winkler \cite{Winkler:1988} on the extreme points of moment constraints,  facilitate
a Monge-Wasserstein linear programming
representation of supersets of the extreme points which can be used for convex maximization over the
Kantorovich  or Prokhorov ball about a measure whose support has at most $n$ points.
A stronger  application of  Winkler \cite[Thm.~2.1]{Winkler:1988} is then used to more fully develop
 these representations  in Section \ref{sec_computation} towards the goal of their efficient computation.
Finally, in Section  \ref{sec6} we consider  when the central measure is an empirical measure.

\section{Main Results}
For a metric space $(X,d)$, the Prokhorov metric $d_{Pr}$ on the space
$\mathcal{M}(X)$ of Borel probability measures is defined
by
\begin{equation}
\label{def_Pr}
  d_{Pr}(\mu_{1},\mu_{2}):=\inf{\bigl\{\epsilon: \mu_{1}(A) \leq \mu_{2}(A^{\epsilon})+\epsilon,\, A \in \mathcal{B}(X)\bigr\} },
\quad \mu_{1},\mu_{2}\in \mathcal{M}(X)\, ,
\end{equation}
where
\[ A^{\epsilon}:=\{x' \in X: d(x,x') < \epsilon \, \, \, \text{for some} \, \, \,  x \in A\} \,  .\]
According to Dudley \cite[Thm.~11.3.3]{Dudley:2002}, when $X$ is separable
the Prokhorov metric metrizes weak convergence. Note that this definition produces the same metric
if we were to use the ``closed" inflated sets $ A^{\epsilon}:=\{x' \in X: d(x,x') \leq  \epsilon \, \,  \text{for some} \, \,  x \in A\} \, $ instead.
On the other hand, the  Kantorovich distance $d_{K}$ on the space  $\mathcal{M}(X)$ of Borel probability measures
on a separable metric space $X$ is defined as follows,
see Vershik \cite{vershik2006kantorovich} for a historical review:
Let \[ \snorm{f}_{L}:=\sup_{x_{1}\neq x_{2}}{\frac{|f(x_{1})-f(x_{2})|}{d(x_{1},x_{2})}} \, \]
denote the Lipschitz norm of a real valued function on $X$. Then the Kantorovich distance is defined by
\begin{equation}
\label{def_kantorovich}
d_{K}(\mu_{1},\mu_{2}):=\sup_{\snorm{f}_{L}\leq1}{\int{fd(\mu_{1}-\mu_{2})}}\, .
\end{equation}
 According to the remark after
 \cite[Lem.~11.8.3]{Dudley:2002}, $d_{K}$ is an extended metric on $\mathcal{M}(X)$.
Let $\Delta_{n}(X) \subset \mathcal{M}(X)$ denote the set of probability measures whose supports have at most
 $n$ points, and let $\ext(A)$ denote the set of extreme points of a set $A$.
We can now state our result for the Prokhorov metric  and Kantorovich extended metric.
  For either of these 
$\hat{d}:=d_{K}$ or $\hat{d}:=d_{Pr}$, for  
$\mu \in \mathcal{M}(X)$ we define $B_{\epsilon}(\mu_{n}):=\{\mu' \in \mathcal{M}(X):\hat{d}(\mu',\mu)\leq \epsilon\}$.\\
\begin{thm}
\label{thm_wozabal}
Let $X$ be a Borel subset of a Polish  metric space and consider the space $\mathcal{M}(X)$ of Borel probability measures equipped with the Prokhorov metric
or the Kantorovich extended metric. For $n \in \mathbb{N}$,  $\epsilon >0$ and $\mu_{n} \in \Delta_{n}(X)$,
 consider the 
ball  $B_{\epsilon}(\mu_{n})$  about the
measure $\mu_{n}$. Then
\[  \ext\bigl(B_{\epsilon}(\mu_{n})\bigr) \subset \Delta_{n+2}(X)\, . \]
\end{thm}

Our path to Theorem \ref{thm_wozabal} requires the development of more useful results which we now describe.
 At the heart of the matter
 is a result of Winkler 
regarding the existence of extreme points of closed convex sets of probability measures
 that is
implicit in the results of Winkler \cite{winkler1978integral,Winkler:1988}.
 Since this result is more modest
than Winkler's goal of developing integral representations, the proof we present
 appears somewhat simpler, in particular
it is different
 in that it does
not utilize Lusin's Theorem.\\
\begin{thm}[Winkler]
\label{thm_winkler}
Let $X$ be a Borel subset of a Polish metric space and consider the set $\mathcal{M}(X)$ of probability measures
equipped with the weak topology. Then every nontrivial closed convex subset of  $\mathcal{M}(X)$
 has an extreme point.
\end{thm}
Winkler's Theorem \ref{thm_winkler}
 is fundamental in the proof of our second main result, the following Theorem \ref{thm_wasser}, regarding
  the extreme points of the
Monge-Wasserstein distance. This result combined with the duality results of Strassen and Kantorovich-Rubinstein
  are then used to establish Theorem \ref{thm_wozabal}. Moreover, in Section \ref{sec_computation}, 
 Corollary \ref{cor_wasser}
to Theorem \ref{thm_wasser} establishes the main results
  to be used towards
  the  computation of
 supersets  of the extreme points $\ext\bigl(B_{\epsilon}(\mu_{n})\bigr)$,
 useful for convex maximization, in particular linear programming,  over the ball $B_{\epsilon}(\mu)$.

For any two probability measures $\mu_{1}, \mu_{2} \in \mathcal{M}(X)$, let
$M(\mu_{1},\mu_{2}) \subset \mathcal{M}(X\times X)$ denote those probability measures with marginals
$\mu_{1}$ and $\mu_{2}$. Then for a non-negative lower semicontinuous real-valued cost function
$c:X\times X\rightarrow \R$, the  Monge-Wasserstein distance $d_{W}$ on $ \mathcal{M}(X)$
is defined by
\[d_{W}(\mu_{1},\mu_{2}):=\inf_{\nu \in M(\mu_{1},\mu_{2})}{\int{c(x,x')d\nu(x,x')}}\, .\]
Let  $ P_{1}:\mathcal{M}(X\times X)\rightarrow \mathcal{M}(X)$ denote  the marginal map corresponding to the first component and $ P_{2}$ the marginal map with respect to the second component.\\
\begin{thm}
\label{thm_wasser}
Let $X$ be a Borel subset of a Polish  metric space and $c:X\times X \rightarrow \R$ a non-negative real-valued lower semicontinuous function.   For $n \in \mathbb{N}$,  $\epsilon >0$ and $\mu_{n} \in \Delta_{n}(X)$, consider the subset
  \[\Gamma_{\mu_{n},\epsilon}:=\{\nu \in \mathcal{M}(X\times X): P_{1}\nu=\mu_{n}, \, \int{c(x,x')d\nu(x,x')}\leq \epsilon\}\, .\]
Then
 \[\ext(\Gamma_{\mu_{n},\epsilon}) \subset \Delta_{n+2}(X\times X)\, \]
and
\[ P_{2}\bigl(\ext(\Gamma_{\mu_{n},\epsilon})\bigr) \supset \ext\bigl(P_{2}(\Gamma_{\mu_{n},\epsilon})\bigr) \, . \]
In particular, we have
\[\ext\bigl(P_{2}(\Gamma_{\mu_{n},\epsilon})\bigr) \subset  \Delta_{n+2}(X)\, .\]
\end{thm}

\section{Computation of supersets}
\label{sec_computation}
Now we show how
the duality results of Strassen and Kantorovich-Rubinstein combined with Theorem \ref{thm_wasser} can be used
in the  computation of
supersets of the extreme points of  $B_{\epsilon}(\mu_{n})$.
 To begin we introduce some terminology.
We say that a set $B$ is a {\em superset for} $B_{\epsilon}(\mu_{n})$ if
\begin{equation}
\label{def_superset}
\ext\bigl(B_{\epsilon}(\mu_{n})\bigr) \,\, \subset\,  B\, \subset\,\,  B_{\epsilon}(\mu_{n})\, .
\end{equation}
For any function $F$ which achieves its maximum at the extreme points, that is
\[\max_{\mu \in B_{\epsilon}(\mu_{n})}{F(\mu)}= \max_{\mu \in \ext(B_{\epsilon}(\mu_{n}))}{F(\mu)}\, ,\]
it follows that
\[\max_{\mu \in B_{\epsilon}(\mu_{n})}{F(\mu)}= \max_{\mu \in B}{F(\mu)}\, \]
for any superset $B$ for $B_{\epsilon}(\mu_{n})$.  Consequently, efficiently constructed supersets
 facilitate the efficient solution to optimization problems
over $B_{\epsilon}(\mu_{n})$.
To fix terms, we restrict our attention to the Prokhorov case, the Kantorovich case being essentially the same.
For fixed $\epsilon >0$ and $\mu_{n}\in \Delta_{n}$, let  us consider the  Prokhorov ball $B_{\epsilon}(\mu_{n})$.
Then it is clear that since $\ext\bigl(B_{\epsilon}(\mu_{n})\bigr) \subset B_{\epsilon}(\mu_{n})$
we obtain from Theorem \ref{thm_wozabal} that
\[ \ext\bigl(B_{\epsilon}(\mu_{n})\bigr) \subset  B_{\epsilon}(\mu_{n}) \cap \Delta_{n+2}(X)\,. \]
Since moreover,
$\ext\bigl(B_{\epsilon}(\mu_{n})\bigr) \subset \partial B_{\epsilon}(\mu_{n})$, where
$\partial B_{\epsilon}(\mu_{n}):=\{\mu \in \mathcal{M}(X): d_{Pr}(\mu,\mu_{n})=\epsilon\}$ is the sphere,
we also conclude that
\[ \ext\bigl(B_{\epsilon}(\mu_{n})\bigr) \subset  \partial B_{\epsilon}(\mu_{n})\cap \Delta_{n+2}(X)\,  .\]
However,  these supersets may be difficult to compute, so we look to Theorem  \ref{thm_wasser}
for sets generated by linear programming.
To that end,  write $\{d>\epsilon\}$ for the subset of elements $(x,y)\in X\times X$ such that $d(x,y)>\epsilon$, and consider
 the subset
 $\Gamma_{\mu_{n},\epsilon}\subset \mathcal{M}(X\times X)$ defined in the proof of Theorem \ref{thm_wozabal} by
\[ \Gamma_{\mu_{n},\epsilon}:= \Bigl\{\nu \in \mathcal{M}(X\times X):  \nu\{d>\epsilon\}\leq \epsilon,
\, P_{1}\nu=\mu_{n}\Bigr\}\, .\]
The proof of Theorem \ref{thm_wozabal} used Strassen's Theorem to assert  in \eqref{strassen} that
\begin{equation*}
  P_{2}\bigl(\Gamma_{\mu_{n},\epsilon}\bigr)=B_{\epsilon}(\mu_{n})\, .
\end{equation*}
Then Theorem \ref{thm_wasser}  implies
\begin{equation}
\label{e_r1}
 \ext(\Gamma_{\mu_{n},\epsilon}) \subset \Delta_{n+2}(X\times X)
\end{equation}
and the string of inequalities
\begin{eqnarray*}
\label{e_r2}
\ext\bigl(B_{\epsilon}(\mu_{n})\bigr) &=& \ext\bigl(P_{2}(\Gamma_{\mu_{n},\epsilon})\bigr)\\
&\subset& P_{2}\bigl(\ext(\Gamma_{\mu_{n},\epsilon})\bigr)\\
&\subset& \Delta_{n+2}(X)\, .
\end{eqnarray*}
Consequently, we obtain
\begin{cor}
\label{cor_wasser}
Consider the situation of Theorem \ref{thm_wozabal} and the set $\Gamma_{\mu_{n},\epsilon}$ defined in Theorem
\ref{thm_wasser}  by
$c:=d$ in the Kantorovich case and $c:=\eins_{d>\epsilon}$ in the Prokhorov case.
Then we have
\begin{eqnarray*}
\ext\bigl(B_{\epsilon}(\mu_{n})\bigr)\quad \subset & P_{2}\bigl(\ext(\Gamma_{\mu_{n},\epsilon}) \bigr) &
 \subset \quad  B_{\epsilon}(\mu_{n})\cap \Delta_{n+2}(X)\\
\ext\bigl(B_{\epsilon}(\mu_{n})\bigr)\quad \subset & P_{2}\bigl(\Gamma_{\mu_{n},\epsilon} \cap \Delta_{n+2}(X\times X)\bigr)
&\subset \quad  B_{\epsilon}(\mu_{n})\cap \Delta_{n+2}(X)\,  .
\end{eqnarray*}
\end{cor}
The statement of Corollary \ref{cor_wasser} captures the mechanism by which we obtain
the improvement from $n+3$ to $n+2$ Dirac masses in the description of the extreme points
in Theorem \ref{thm_wozabal}. Indeed, since the set $\Gamma_{\mu_{n},\epsilon}$ is
 a set of measures subject to $n+1$ constraints, its extreme points are convex combination
of $n+2$ Dirac masses on the product space $X\times X$. Then the fact that
the extreme points of $B_{\epsilon}(\mu_{n})$ consists of the convex combination of 
$n+2$ Dirac masses follows from the fact that Corollary
\ref{cor_wasser} implies that the projection onto  the second component of these extreme points covers all  the 
extreme
points of $B_{\epsilon}(\mu_{n})$, and the fact that 
projection of Dirac masses on $X\times X$ are Dirac masses
on $X$. 

Corollary \ref{cor_wasser} also says that both 
\[ P_2\big(\Gamma_{\mu_{n},\epsilon} \cap  \Delta_{n+2}(X\times X) \big)\quad  \text{and}\quad
P_2\big(\ext(\Gamma_{\mu_{n},\epsilon})\big)\]
 are supersets
for $ B_{\epsilon}(\mu_{n})$.
Although the latter is smaller in that
\[  P_{2}\bigl(\ext(\Gamma_{\mu_{n},\epsilon})\bigr) \subset P_{2}\bigl(\Gamma_{\mu_{n},\epsilon} \cap  \Delta_{n+2}(X\times X)\bigr)\,,\]
 the computation of the former is useful in the computation of the latter,
so we consider the computation of both.

\subsection{Computing $\Gamma_{\mu_{n},\epsilon} \cap  \Delta_{n+2}(X\times X)$}
\label{sec_comp_p1}
Since, by  \eqref{e_r1}, both  $\ext(\Gamma_{\mu_{n},\epsilon})$
 and
$ \Gamma_{\mu_{n},\epsilon} \cap  \Delta_{n+2}(X\times X) $ are subsets of
$P_{1} ^{-1}\mu_{n}\cap \Delta_{n+2}(X\times X)$,  it will be convenient to compute $P_{1}^{-1}\mu_{n}\cap \Delta_{n+2}(X\times X)$
first. Let us proceed inductively, and assume that $\mu_{n} \in \Delta_{n}(X)$ but is not in $\Delta_{n-1}(X)$.
 Then
$\mu_{n}:=\sum_{i=1}^{n}{\beta_{i}\d_{y_{i}}}$   with $ \beta_{i}> 0, y_{i} \in X, i=1,..,n,
\, \, \sum_{i=1}^{n}{\beta_{i}}=1$, and
 $y_{i}\neq y_{j},i\neq j$.
Fixing this $y=(y_{i})$ and $(\beta_{i})$,
we now define some subsets of $\mathcal{M}(X\times X)$. For $x \in X^{m}, n \leq m \leq n+2$,
denote
\[\delta_{y,x}:= \sum_{k=1}^{n}{\beta_{k} \d_{y_{k},x_{k}}}\, ,\]
 and let
\begin{equation}
\label{def_pi0}
  \Pi_{0}
:=\bigl\{\delta_{y,x}\, x \in X^{n} \bigr\}\, .
\end{equation}
For $i=1,..,n$ and $x \in X^{n+1}$, define
 \begin{equation}
\label{def_p1x}
  \Pi_{i}(x)
:=\delta_{y,x}+\bigl\{\gamma(\d_{y_{i},x_{n+1}}- \d_{y_{i},x_{i}})
\,,\, 0 < \gamma < \beta_{i} \bigr\}\,
\end{equation}
and
\begin{equation}
\label{def_p1}
  \Pi_{i}:=\{\Pi_{i}(x), x \in X^{n+1}\}\, .
\end{equation}
Moreover, for $x \in X^{n+2}$ and
for $i<j$, define
\begin{equation}
\label{def_p2ax}
 \Pi_{i,j}(x):= \delta_{y,x}+
\Bigl\{\gamma_{i}(\d_{y_{i},x_{n+1}}-\d_{y_{i},x_{i}})
+\gamma_{j}(\d_{y_{j},x_{n+2}}-\d_{y_{j},x_{j}})\, ,\\
 0 < \gamma_{i}<\beta_{i}, 0 <  \gamma_{j}< \beta_{j} \Bigr\}
\end{equation}
while for
$i=j$, define
\begin{equation}
\label{def_p2bx}
 \Pi_{i,i}(x):=\delta_{y,x}+
\Bigl\{\gamma_{1}(\d_{y_{i},x_{n+1}} -\d_{y_{i},x_{i}})
 +
\gamma_{2}(\d_{y_{i},x_{n+2}}-\d_{y_{i},x_{i}}),\,\,\\
 \gamma_{1}>0, \gamma_{2}>0, \gamma_{1}+\gamma_{2} <  \beta_{i}\Bigr\}
\end{equation}
and then, for $i\leq j$, again take the union
\begin{equation}
\label{def_p2}
  \Pi_{i,j}:=\{\Pi_{i,j}(x), x \in X^{n+2}\}\, .
\end{equation}
\begin{lem}
\label{lem_eee}
In terms of the sets defined in \eqref{def_pi0}, \eqref{def_p1}, and \eqref{def_p2}, we have
\[P_{1}^{-1}\mu_{n}\cap \Delta_{n+2}(X\times X)=\Pi_{0}  \cup_{k=1}^{n}\Pi_{k} \cup_{i\leq j} \Pi_{i,j}\, .\]
\end{lem}

Using Lemma \ref{lem_eee}, we
can now obtain an almost explicit representation of
$ \Gamma_{\mu_{n},\epsilon} \cap  \Delta_{n+2}(X\times X)$, almost in the sense
that it will amount to an explicitly represented set subject to the constraint of a
  single explicitly computable function. To that end, let us combine the definitions
\eqref{def_pi0}, \eqref{def_p1}, and \eqref{def_p2} of $\Pi_{0}$,  $\Pi_{i}$ and $\Pi_{i,j}$ into one symbol with the introduction of a multiindex
  $\imath$
that can take the values $\imath=0$,  $ \imath =i$ for $i\in \{1,n\}$, or $\imath = (i,j)$ with $i\leq j$. Then, in this notation
$\Pi_{\imath}(x)$ will denote $\Pi_{0}(x)$ and imply $x \in X^{n}$ when $\imath=0$, it will
denote $\Pi_{i}(x)$ and imply $x \in X^{n+1}$ when $\imath=i$, and
denote $\Pi_{i,j}(x)$ and imply $x \in X^{n+2}$  when $\imath=(i,j)$.

  Since,
in general,  for
 $\nu:=\sum_{k=1}^{m}{\alpha_{k} \d_{x_{k},x'_{k}}}$ we have
\begin{equation}
\label{explicit}
  \nu\{d>\epsilon\}=
\sum_{k=1}^{m}{\alpha_{k}\eins_{d(x_{k},x'_{k})>\epsilon}}\, ,
\end{equation}
 it follows that the function
$ \nu \mapsto \nu\{d>\epsilon\}$ restricted to $  \Delta_{n+2}(X\times X)$ is explicitly computable.
Then, since
\begin{equation}
\label{id8}
\Gamma_{\mu_{n},\epsilon} \cap  \Delta_{n+2}(X\times X)=P_{1}^{-1}\mu_{n}\cap \Delta_{n+2}(X\times X)
\cap \bigl\{ \nu \in \mathcal{M}(X\times X): \nu\{d>\epsilon\} \leq \epsilon\bigr\}\, ,
\end{equation}
if we
 incorporate the constraint $ \nu\{d>\epsilon\} \leq \epsilon$
by defining
\begin{equation}
\label{def_barpi0}
  \bar{\Pi}_{\imath}(x):=\Pi_{\imath}(x)\cap \bigl\{ \nu \in \mathcal{M}(X\times X): \nu\{d>\epsilon\} \leq \epsilon\bigr\}\, ,
\end{equation}
along with their unions $\ \bar{\Pi}_{\imath}$ over $X^{n}$, $X^{n+1}$ and $X^{n+2}$ respectively,
then from the distributive law of set theory,  Lemma \ref{lem_eee} and \eqref{id8}, we conclude that
\begin{equation}
\label{id_GammaDelta}
\Gamma_{\mu_{n},\epsilon} \cap  \Delta_{n+2}(X\times X)=\bar{\Pi}_{0}  \cup_{k=1}^{n}\bar{\Pi}_{k} \cup_{i\leq j}
 \bar{\Pi}_{i,j} \, .
\end{equation}

\subsection{Computing $\ext(\Gamma_{\mu_{n},\epsilon})$}
\label{sec_comp_p2}
To compute $\ext(\Gamma_{\mu_{n},\epsilon})$ we use a stronger version of the characterization
of the extreme points found in Winkler \cite[Thm.~2.1]{Winkler:1988} than we used in Theorem
\ref{thm_wasser}, along with the computation of $P_{1}^{-1}\mu_{n}\cap \Delta_{n+2}(X\times X)$ from
Lemma \ref{lem_eee}. To that end, consider the constraint functions
$f_{i}:=\eins_{y_{i}\times X}, i=1,..,n$ (where $\eins_{y_{i}\times X}(a,b)=1$ if $a=y_i$ and $\eins_{y_{i}\times X}(a,b)=0$ if $a\not=y_i$) and
$f_{n+1}:=\eins_{d>\epsilon}$. Then Winkler's \cite[Thm.~2.1]{Winkler:1988} assertion
\begin{eqnarray*}
\ext(\Gamma_{\mu_{n},\epsilon})& \subset &  \Bigl\{\nu \in \Gamma_{\mu_{n},\epsilon}:\nu=\sum_{i=1}^{m}{\alpha_{i}\d_{x_{i},x'_{i}}}, 1 \leq m \leq n+2,
\alpha_{i} >0, x_{i},x'_{i} \in X, i=1,..,m,\\
&& \text{the vectors}\,\, \bigl(f_{1}(x_{i},x'_{i}), \ldots, f_{n+1}(x_{i},x'_{i}),1\bigr), \, i=1,..,m\,\,
\text{are linearly independent}
 \Bigr\}
\end{eqnarray*}
 amounts to
\begin{equation}
\label{winkler2}
\ext(\Gamma_{\mu_{n},\epsilon}) \subset   \Bigl\{\nu \in \Gamma_{\mu_{n},\epsilon}:\nu=\sum_{i=1}^{m}{\alpha_{i}\d_{x_{i},x'_{i}}}, 1 \leq m \leq n+2,
\alpha_{i} >0, x_{i},x'_{i} \in X, i=1,..,m,
\end{equation}
\[
 \text{the vectors}\,\, \bigl(\eins_{y_{1}}(x_{i}),..., \eins_{y_{n}}(x_{i}), \eins_{d(x_{i},x'_{i})>\epsilon},1\bigr), \, i=1,..,m\,\,
\text{are linearly independent}
 \Bigr\}\, .
\]

Since  Theorem \ref{thm_wasser} asserts that $\ext(\Gamma_{\mu_{n},\epsilon}) \subset \Delta_{n+2}(X \times X)$,
 it follows
that we can replace $\Gamma_{\mu_{n},\epsilon}$ by $\Gamma_{\mu_{n},\epsilon}\cap \Delta_{n+2}(X\times X)$
in the righthand side of \eqref{winkler2}. Having done so,
let us define
\begin{eqnarray}
\label{winkler3}
\bar{\Theta}&:= &  \Bigl\{\nu \in \Gamma_{\mu_{n},\epsilon}\cap \Delta_{n+2}(X\times X):\nu=\sum_{i=1}^{m}{\alpha_{i}\d_{x_{i},x'_{i}}}, 1 \leq m \leq n+2,
\alpha_{i} >0, x_{i},x'_{i} \in X, i=1,..,m,\notag\\
 &&\text{the vectors}\,\, \bigl(\eins_{y_{1}}(x_{i}),..., \eins_{y_{n}}(x_{i}), \eins_{d(x_{i},x'_{i})>\epsilon},1\bigr), \, i=1,..,m\,\,
\text{are linearly independent}
 \Bigr\}\, . \notag\\
\end{eqnarray}
to be the righthand side of
  \eqref{winkler2}. Then we have
\[  \ext(\Gamma_{\mu_{n},\epsilon})  \subset \bar{\Theta} \subset  \Gamma_{\mu_{n},\epsilon}\,  \]
and therefore $\bar{\Theta}$ is a superset for $\Gamma_{\mu_{n},\epsilon}$. To compute it,
for $i \in \{1,..,n\}$, let us define
\begin{equation}
\label{lambda1}
\Lambda_{i}:=\{x\in X^{n+1}: \eins_{d(y_{i},x_{n+1})> \epsilon}\neq \eins_{d(y_{i},x_{i})> \epsilon}  \}\, .
\end{equation}
and for $i<j$ define
\begin{equation}
\label{lambda2}
\Lambda_{i,j}:=\{x\in X^{n+2}: \eins_{d(y_{i},x_{n+1})> \epsilon}\neq \eins_{d(y_{i},x_{i})> \epsilon},\,
 \eins_{d(y_{j},x_{n+2})> \epsilon}\neq \eins_{d(y_{j},x_{j})> \epsilon} \}\, .
\end{equation}
\begin{lem}
\label{lem_eee2}
With $\Lambda_{i}$ defined in \eqref{lambda1}, $\Lambda_{i,j}$ defined in
\eqref{lambda2}, and  $\bar{\Pi}_{0}, \bar{\Pi}_{i}$ and
$\bar{\Pi}_{i,j}$ defined in
\eqref{def_barpi0}, we have
\[\bar{\Theta}= \bar{\Pi}_{0}  \cup_{k=1}^{n}(\bar{\Pi}_{i}\cap  \Lambda_{i}) \cup_{i< j}(\bar{\Pi}_{i,j}\cap  \Lambda_{i,j}) \, . \]
\end{lem}

\begin{rmk}
For a reference measure $\mu:= \sum_{k=1}^{n}{\beta_{k} \d_{y_{k}}}$, it is interesting to note that
 the condition that a  measure
\[ \delta_{y,x}+\bigl\{\gamma(\d_{y_{i},x_{n+1}}- \d_{y_{i},x_{i}})
\,,\, 0 < \gamma < \beta_{i} \bigr\}\, \]
is a member of $\Pi_{i}\cap \Lambda_{i}$  amounts to the  splitting off of the mass $\beta_{i}$ on the Dirac
situated at $y_{i}$ into the convex sum of two Dirac masses, one situated
at $(y_{i},x_{i})$ and one at $(y_{i},x_{n+1})$, such that, between $x_{i}$ and $x_{n+1}$, one is inside
the ball of radius $\epsilon$ about $y_{i}$ and the other is outside it. Moreover, to be a member
of $\Pi_{i,j}$ with $i<j$ amounts to two such splits.
\end{rmk}

\subsection{Equivalence classes determined by the adjacency matrix}
For $x \in X^{m}, n \leq m \leq n+2$, let its adjacency matrix
$A(x)$ be defined by
\[ A^{i,j}(x):=\eins_{d(y_{i},x_{j})>\epsilon},\quad i=1,..,n,\, j=1,..,m\, . \]
Commensurate with our introduction of the multiindex  $\imath$, we use the expression $A(x)$ to mean
the $n\times m$ adjacency matrix when $x \in X^{m}$, for any $m=n,n+1,n+2$.
Since, by Lemma \ref{lem_eee2}, $\bar{\Theta}= \bar{\Pi}_{0}  \cup_{k=1}^{n}\bar{\Pi}_{k} \cup_{i\leq j} \bar{\Pi}_{i,j}$
and the latter are determined by conditions
 $\Lambda_{i}, i=1,..,n$,
 $\Lambda_{i,j}$ for  $i < j$, and
$\nu\{(z,z') \in X\times X: d(z,z')>\epsilon\} \leq \epsilon$, all of which, by the the evaluation  \eqref{explicit},
 only depend on the values of the adjacency matrix,  we obtain
 the following lemma.
It asserts that, for any point in $\bar{\Pi}_{0}$, $\bar{\Pi}_{i}$ or $\bar{\Pi}_{i,j}$,
if the second components $x$ of the Dirac masses are changed to $x'$ with the same adjacency matrix, then
the resulting sum of Dirac masses remains in  $\bar{\Pi}_{0}$, $\bar{\Pi}_{i}$ or $\bar{\Pi}_{i,j}$ respectively. Consequently,
it will be useful in the efficient exploration of the set $\bar{\Theta}$.
\begin{lem}
\label{lem_adjacency}
For $n\leq m \leq n+2$,  $x \in X^{m}$,  $ z\in X^{m}$ and $\alpha \in \R^{m}$,
consider $\mu(x):=\sum_{k=1}^{m}{\alpha_{k} \d_{z_{k},x_{k}}}$. If
 $\mu(x) \in \bar{\Pi}_{\imath}(x)$, then
 for all $x'$ such that $A(x')=A(x)$, we have
$\mu(x') \in \bar{\Pi}_{\imath}(x')$.
\end{lem}

\section{Extreme points of a ball about an empirical measure}\label{sec6}
Empirical measures take the form
 $\mu_{n}=\frac{1}{n}\sum_{i=1}^{n}{\d_{y_{i}}}$, with $ y_{i}\in X, i=1,..,n$. When all the points
$y_{i}$ are unique, we can define
 $\beta_{i}:=\frac{1}{n}, i=1,..,n$ in the expressions of Section \ref{sec_computation},
when the points have duplicates things will be more complicated.  In the unique case, the definitions
\eqref{def_pi0}, \eqref{def_p1x}, \eqref{def_p2ax} and \eqref{def_p2bx}  of
$\Pi_{0}$, $\Pi_{i}(x)$ and $\Pi_{i,j}(x)$ take on a more symmetrical form, and since the case when
the central measure is an empirical measure is  an important application,
we spell them out.
To begin with, we have
\begin{equation*}
\label{def_de}
\delta_{y,x}=\frac{1}{n}\sum_{k=1}^{n}{\d_{y_{k},x_{k}}}\, .
\end{equation*}
 Moreover, the evaluation
of the constraint $\nu(d>\epsilon)\leq \epsilon$ also takes a simpler form, so that
constrained sets $\bar{\Pi}_{0}$, $\bar{\Pi}_{i}(x)$ and $\bar{\Pi}_{i,j}(x)$ appear as follows:
\begin{equation*}
\label{def_barpi0e}
  \bar{\Pi}_{0}
=\bigl\{  \delta_{y,x},\, x \in X^{n} \bigr\}\,
\end{equation*}
subject to the constraint
\begin{equation*}
\label{def_barpi0ec}
\frac{1}{n}\sum_{k=1}^{n}{\eins_{d(y_{k},x_{k})>\epsilon}}
 \leq \epsilon  \, ,
\end{equation*}
while for $i \in \{1,..,n\}$ we have
\begin{equation*}
\label{def_barp1xe}
 \bar{\Pi}_{i}(x)
=\delta_{y,x}+\frac{1}{n}\bigl\{\gamma(\d_{y_{i},x_{n+1}}- \d_{y_{i},x_{i}})
\,,\, 0 < \gamma < 1 \bigr\}\, \,
\end{equation*}
subject to the constraint
\begin{equation*}
\label{def_barp1xec}
\frac{1}{n}\sum_{k=1}^{n}{\eins_{d(y_{k},x_{k})>\epsilon}}+\gamma(\eins_{d(y_{i},x_{n+1})>\epsilon}-
\eins_{d(y_{i},x_{i})>\epsilon}) \leq \epsilon\, ,
\end{equation*}
and  for $i<j$ we have
\begin{equation*}
\label{def_barp2ax}
 \bar{\Pi}_{i,j}(x)= \delta_{y,x}+ \frac{1}{n}
\Bigl\{\gamma_{i}(\d_{y_{i},x_{n+1}}-\d_{y_{i},x_{i}})
+\gamma_{j}(\d_{y_{j},x_{n+2}}-\d_{y_{j},x_{j}})\, ,\\
 0 < \gamma_{i}< 1, 0 <  \gamma_{j}<  1 \Bigr\}
\end{equation*}
subject to the constraint
\begin{equation*}
\label{def_barp2axc}
\frac{1}{n}\sum_{k=1}^{n}{\eins_{d(y_{k},x_{k})>\epsilon}}+
\gamma_{i}(\eins_{d(y_{i},x_{n+1})>\epsilon}-\eins_{d(y_{i},x_{i})>\epsilon})+
\gamma_{j}(\eins_{d(y_{j},x_{n+2})>\epsilon}-\eins_{d(y_{j},x_{j})>\epsilon})\,  \leq \epsilon\, ,
\end{equation*}
and for $i=j$ we have
\begin{equation*}
\label{def_barp2bx}
\bar{\Pi}_{i,i}(x)=
\delta_{y,x}+
\Bigl\{\gamma_{1}(\d_{y_{i},x_{n+1}} -\d_{y_{i},x_{i}})
 +
\gamma_{2}(\d_{y_{i},x_{n+2}}-\d_{y_{i},x_{i}}),\,\,\\
 \gamma_{1}>0, \gamma_{2}>0, \gamma_{1}+\gamma_{2} <  1\Bigr\}
\end{equation*}
subject to the constraint
\begin{equation*}
\label{def_barp2bxc}
\frac{1}{n}\sum_{k=1}^{n}{\eins_{d(y_{k},x_{k})>\epsilon}}+
\gamma_{1}(\eins_{d(y_{i},x_{n+1})>\epsilon}-\eins_{d(y_{i},x_{i})>\epsilon})
+\gamma_{2}(\eins_{d(y_{i},x_{n+2})>\epsilon}-\eins_{d(y_{i},x_{i})>\epsilon})\, \leq \epsilon \,  .
\end{equation*}

\section{Appendix}

\subsection{Extreme subsets}
We begin by establishing a fundamental identity regarding the extreme subsets of extreme subsets\footnote{
the repetition here is not a typo} of an affine space.
 Since this terminology varies in the literature, we fix it now. Following \cite[Def.~7.61]{AliprantisBorder:2006}, we say that a set $E$ is an {\em extreme subset}
 of a subset $A\subset L$
of a real linear space $L$ if  $E \subset A$ and $\theta x +(1-\theta)y \in E$ with $x,y \in A,\, \theta \in (0,1)$, implies that
$x,y \in E$. Note that this definition does not require convexity. An {\em extreme point} of $A$ is an extreme subset of $A$ consisting of a single point. We say that a set $F$ is a {\em face} of a subset $A\subset L$
of a real linear space $L$ if it is a convex extreme subset of $A$.   The following lemma implies that
Simon \cite[Prop.~8.6]{Simon} is valid {\em  without assuming compactness or convexity}.
\begin{lem}
\label{lem_extremesubsets}
Let $A$ be a subset of a real linear space $L$ and let $E$ be an extreme subset of $A$. Then
$B$ is an extreme subset of $E$ if and only if $B\subset E$ and it is an extreme subset of $A$. In particular,
\[ \ext(E) = E\cap \ext(A)\, .\]
\end{lem}

\begin{proof}
The proof is identical to that of \cite[Prop.~8.6]{Simon}, but we reproduce it here so that the reader can confirm
that it is  valid without compactness or convexity assumptions.
First suppose that $B\subset E$ and $B$ is an extreme subset of $A$.  Then, by definition, if
$\theta x +(1-\theta)y \in B$, with $x,y \in A,\, \theta \in (0,1)$, then
$x,y \in B$. Since $E \subset A$, it follows that if we have
$\theta x +(1-\theta)y \in B$, with $x,y \in E, \theta \in (0,1)$,   that
$x,y \in B$. Consequently, since $B\subset E$, $B$ is an extreme subset of $E$. Now assume that $B$ is an extreme subset of $E$.
Then, if we have
$\theta x +(1-\theta)y \in B$, with $x,y \in A, \theta \in (0,1)$, the fact that $B\subset E$
and $E$ is an extreme subset of $A$ implies that $x,y\in E$. Then, since $B$ is an extreme subset of $E$,
it follows that $x,y\in B$. Since clearly $B \subset A$, we conclude that $B$ is an extreme subset of $A$.
\end{proof}

\subsection{Affine images of extreme points}
\label{sec_affineextreme}
Here we establish a fundamental result for affine transformations and extreme points of, possibly non-convex, subsets.
\begin{lem}
\label{lem_affineedtreme2}
Let $L$ and $L'$ be real linear spaces and
 $K\subset L$  a subset. Suppose that
  $G:K \rightarrow L'$ is the restriction of an affine transformation $G:L \rightarrow L'$  to  $K$
 such that $\ext(G^{-1}(k'))\neq \emptyset$
 for all $k'\in \ext(G(K))$.
Then $G(\ext(K)) \supset \ext(G(K))$.
\end{lem}

\begin{proof}
Let $k' \in \ext(G(K))$ and consider
 any point $k \in G^{-1}(k')$. Then if
$k=\theta k_{1} +(1-\theta)k_{2}$, with $k_{1},k_{2} \in K, \, \theta \in (0,1)$, then
$k'= G(k)=G(\theta k_{1} +(1-\theta)k_{2} )=\theta G(k_{1})+(1-\theta) G(k_{2})$, so that,  since
$k'$ is an extreme point, it follows that $G(k_{1})=G(k_{2})=G(k)$. That is, $G^{-1}(k')$ is an extreme subset of $K$.
Therefore,  Lemma \ref{lem_extremesubsets} implies that
\[\ext\bigl(G^{-1}(k')\bigr)=G^{-1}(k')\cap \ext(K)\, ,\]
so that  any extreme point of $G^{-1}(k')$ is an extreme point of $K$.
Since, by assumption, $G^{-1}(k')$ has an extreme point, it follows that any such extreme point is
 an extreme point of $K$. Since the image under $G$ of any such point is $k'$, and $k' \in \ext(G(K))$
was arbitrary, the assertion follows.
\end{proof}

\subsection{Integrals of extended real-valued lower semicontinuous functions}
Here we formulate a  generalization to extended real-valued functions of \cite[Thm.~15.5]{AliprantisBorder:2006}, that the integral of a bounded lower semicontinuous function forms a lower semicontinuous function in the weak topology.
\begin{lem}
\label{lem_lsc}
Let $(X,d)$ be a metric space and $f:X \rightarrow \bar{R}_{+}$ a nonnegative lower semicontinuous extended real-valued
function. For $\mu \in \mathcal{M}(X)$ define $\int{fd\mu}$ to be the integral if $f$ is $\mu$-integrable and
$\infty$ if it is not. Then the function $F:\mathcal{M}(X) \rightarrow \bar{R}$ defined by
$F(\mu):=\int{fd\mu}$ is lower semicontinuous in the weak topology.
\end{lem}
\begin{proof}
We follow   Aliprantis and Border \cite[Thm.~15.5]{AliprantisBorder:2006}.
First let us clip the function $f$ at the level $s$ by $f^{s}(x):=\min{(f(x),s)}, x \in X$. Then since
for all $c$ we have
$\{x:f^{s}(x) \leq c\} =\{x:f(x) \leq c\} $ for $s >c$ and $\{x:f^{s}(x) \leq c\} =\{x:f(x) \leq s\}  $ for
$s \leq c$ it follows that $f^{s}$ is a real-valued semicontinuous function. Consequently, by
 \cite[Thm.~3.13]{AliprantisBorder:2006}  for each $s$, $f^{s}$ is the increasing pointwise limit of a sequence
$f_{n}^{s}$ of Lipschitz continuous functions. By further clipping from below at $0$,
  sending $f_{n}^{s} \mapsto \max{(f^{s}_{n},0)}$ we
obtain that we can assume that for each $s$, $f^{s}$ is the increasing pointwise limit of a sequence
$f_{n}^{s}$ of nonnegative  bounded continuous functions. Therefore, setting $s:=n$ and  defining
$f_{n}:=f_{n}^{n}$, we conclude that
$f$ is the increasing pointwise limit of a sequence $f_{n}$ of bounded continuous nonnegative real-valued functions.

Now let $\mu_{\alpha}$ be a net such that $\mu_{\alpha}
\rightarrow   \mu$  in the weak topology and let us utilize the integration theory for extended real-valued functions as found in Ash \cite[Sec.~1]{Ash:1972}.
Then it follows that
\begin{equation}
\label{e2}
\int{f_{n}d\mu_{\alpha}} \xrightarrow{\alpha} \int{f_{n}d\mu}
\end{equation}
and
\begin{equation}
\label{e2}
\int{f_{n}d\mu_{\alpha}} \leq \int{fd\mu_{\alpha}}
\end{equation}
so that we conclude that
\[\int{f_{n}d\mu} \leq \lim \inf_{\alpha}{\int{fd\mu_{\alpha}}}\, ,\]
for each $n$. Therefore, from the monotone convergence theorem for extended valued functions, see e.g.~Ash
 \cite[1.6.2]{Ash:1972}, we have
\[ \int{fd\mu}=\lim_{n\rightarrow \infty}{\int{f_{n}d\mu}} \] and
 we conclude that
\[\int{fd\mu}=\lim_{n\rightarrow \infty}{\int{f_{n}d\mu}} \leq  \lim \inf_{\alpha}{\int{fd\mu_{\alpha}} }\, , \]
so that the assertion follows from the alternative characterization of lower semicontinuous  extended real-valued
 functions
 \cite[Lem.~2.42]{AliprantisBorder:2006}.
\end{proof}

\section{Proofs}

\subsection{Proof of Theorem \ref{thm_winkler}}
We follow the proof of the main result in  \cite{winkler1978integral}, simplifying it according to our more modest goal.
Let $t$ denote the topology of $X$.
Since $X$ is
a Borel subset of a Polish space, it follows that it is Suslin and therefore all finite Borel  measures
on $(X,t)$ are tight.
  Let $C\subset \mathcal{M}(X)$ be a nontrivial closed convex subset and consider $\mu^{*} \in C$.
Since $\mu^{*}$ is tight, using
a recursive argument, we obtain a sequence $K_{n} \subset X, n \in \mathbb{N}$ of  disjoint compact subsets
such that if we define
$X_{1}:=\cup_{n \in \mathbb{N}}{K_{n}}$ we have
$\mu^{*}(X_{1})=1$.
Let the relative topology of
  the subspace $X_{1} \subset X$ be denoted by $t_{0}$ and introduce a finer topology $t_{1}\supset t_{0}$
defined by
$A \in t_{1}$ if, for every $n \in \mathbb{N}$, we have   $A\cap K_{n}=B_{n} \cap K_{n}$ for some $B_{n} \in t$.
It follows that $K_{n} \in t_{1}$ for all $n \in \mathbb{N}$, so that $(X_{1},t_{1})$ is locally compact. Moreover,
since $(X_{1},t_{0})$ is metric, it is Hausdorff, and since $t_{1}$ is finer than $t_{0}$ it follows that
$(X_{1},t_{1})$ is Hausdorff.
Let us show that $(X_{1},t_{1})$ is also completely regular. To that end, recall,
 see e.g.~Willard  \cite[Thm.~14.12]{Willard}, that
  a space is completely regular if and only if
 its topology is the initial topology
corresponding to the bounded continuous functions.
Since $(X_{1},t_{0})$ is metric it is completely regular.  Consequently
 the topology $t_{1}$ amounts  to the initial topology corresponding to the addition of the set of indicator functions
$\eins_{K_{n}}, n \in \mathbb{N}$ to the collection of continuous functions on $(X_{1},t_{0})$. Therefore,
 $(X_{1},t_{1})$ is also completely regular.
 Since
 $(X,t)$ is Suslin it is second countable and therefore $(X_{1},t_{0})$ is second countable.
Since a base for the topology $t_{1}$ can be constructed by taking a base for $(X_{1},t_{0})$  and  taking all
intersections with the sets $K_{n}, n \in \mathbb{N}$, it follows that  $(X_{1},t_{1})$ is second countable.
Consequently, all  the spaces  $(X,t), (X_{1},t_{0})$ and $(X_{1},t_{1})$ are second countable.

Now observe that for $A\in t_{1}$ we have $ A=\cup_{n \in \mathbb{N}}{A\cap K_{n}}$
and for each $n$, we have $ A\cap K_{n}=B_{n} \cap K_{n}$ for some $B_{n} \in t$. Since both $B_{n}$ and $K_{n}$
are in $\mathcal{B}(t)$ it follows  that the intersection is also and therefore
also the countable union $A=\cup_{n\in \mathbb{N}}{A\cap K_{n}}$. That is, $ A\in \mathcal{B}(t)$ and since
$A\subset X_{1}$ it follows that $ A\in \mathcal{B}(t_{0})$. Since $t_{1}$ is finer than $t_{0}$,  we conclude
that
\[  \mathcal{B}(t_{0}) =\mathcal{B}(t_{1})\, \]
and therefore
\begin{equation}
\label{eq_t0t1}  \mathcal{M}(X_{1},t_{0})=\mathcal{M}(X_{1},t_{1})\,
\end{equation}
{\em as sets}.

Since $(X_{1},t_{1})$ is locally compact and Hausdorff, we  consider
 the Alexandroff one-point compactification $(X_{2},t_{2})$ of $(X_{1},t_{1})$.
 Since
 $(X_{1},t_{1})$ is second countable,
 it follows, see e.g.~\cite[Thm.~3.44]{AliprantisBorder:2006}, that the compactification
$(X_{2},t_{2})$ is metrizable.  Consequently, $(X_{2},t_{2})$ is a compact metrizable Hausdorff space, and
so it follows, see e.g.~\cite[Thm.~15.11]{AliprantisBorder:2006}, that $\mathcal{M}(X_{2},t_{2})$ is compact and metrizable. Moreover, since by  e.g.~\cite[Lem.~3.26 \& Thm.~3.28]{AliprantisBorder:2006}, all compact metrizable spaces are separable and therefore second countable, it follows that  $\mathcal{M}(X_{2},t_{2})$  is
second countable.

Define
\[\mathcal{M}_{X_{1}}(X,t)=\{\mu \in \mathcal{M}(X,t): \mu(X_{1})=1\}\]
\[\mathcal{M}_{X_{1}}(X_{2},t_{2})=\{\mu \in \mathcal{M}(X_{2},t_{2}): \mu(X_{1})=1\}\]
where $X_{1}\subset X_{2}$ is the subset identification corresponding to the compactification.
Since both $\mathcal{M}(X,t)$ and $ \mathcal{M}(X_{2},t_{2})$ are second countable, it follows that the subspaces
$\mathcal{M}_{X_{1}}(X,t)$ and $\mathcal{M}_{X_{1}}(X_{2},t_{2})$ are second countable.
Since $(X_{2},t_{2})$ is compact and Hausdorff it follows from \cite[Thm.~17.10 \& Cor.~15.7]{Willard} that
$(X_{2},t_{2})$ is completely regular. Consequently,
  if we let
\begin{eqnarray*}
i^{0}:(X_{1},t_{0})& \rightarrow & (X,t)\\
i^{1}:(X_{1},t_{1})& \rightarrow  &(X_{2},t_{2})
\end{eqnarray*}
denote the two subset injections, then since both $(X_{1},t_{0})$ and $(X_{2},t_{2})$
are completely regular,
Bourbaki \cite[Prop.~8, Sec.~5.3]{Bourbaki2} implies
  that the pushforward maps
\begin{eqnarray*}
&i^{0}_{*}:\mathcal{M}(X_{1},t_{0})\rightarrow  \mathcal{M}_{X_{1}}(X,t) , & \\
&i^{1}_{*}:\mathcal{M}(X_{1},t_{1}) \rightarrow \mathcal{M}_{X_{1}}(X_{2},t_{2}), &
\end{eqnarray*}
are  homeomorphisms,
  Because of the identity
\eqref{eq_t0t1} it is natural
to  define
\[ \iota:\mathcal{M}_{X_{1}}(X,t) \rightarrow \mathcal{M}_{X_{1}}(X_{2},t_{2})\]
by
\[\iota:=i^{1}_{*}(i^{0}_{*})^{-1}\, .\]
Although each component $i^{0}_{*}$ and $i^{1}_{*}$ of $\iota$ is a homeomorphism, since we have
$ \mathcal{M}(X_{1},t_{0})=\mathcal{M}(X_{1},t_{1})$ only
 as sets, $\iota$ may not be a homeomorphism. However, since $t_{1}$ is finer than $t_{0}$ it follows that the identity
map $\acute{\iota}:\mathcal{M}(X_{1},t_{1}) \rightarrow \mathcal{M}(X_{1},t_{0})$ is continuous, and if we more properly
 write
\[\iota:=i^{1}_{*}(\acute{\iota})^{-1}(i^{0}_{*})^{-1}\, \]
as a composition of three maps on topological spaces, it follows from the continuity of $\acute{\iota}$
and the fact that $i^{0}_{*}$ and $i^{1}_{*}$ are  homeomorphisms, that
\begin{equation}
\label{eq_closed}
\iota \,\, \text{is a closed map}\, .
\end{equation}

Now define
\[C_{0}:=C\cap \mathcal{M}_{X_{1}}(X,t)\, \]
\[C_{2}:=\iota C_{0}\, \]
and
\[ \bar{C}_{2}:= \text{the closure of} \, \, C_{2}\, \text{in}\,  \mathcal{M}(X_{2},t_{2})\, .\]
 Since $\iota$
 is affine it follows that $C_{2}$ is convex.
Moreover, since $C_{0}$ is relatively closed in $\mathcal{M}_{X_{1}}(X,t)$  and by \eqref{eq_closed}
$\iota$ is a closed map, it follows that
$C_{2}=\iota C_{0}$
is relatively closed in  $\mathcal{M}_{X_{1}}(X_{2},t_{2})$.
Consequently, there exists a closed set $\acute{C}_{2}\subset \mathcal{M}(X_{2},t_{2})$ such that
$C_{2}=\acute{C}_{2}\cap \mathcal{M}_{X_{1}}(X_{2},t_{2})$. Since it follows that
$\acute{C}_{2} \supset C_{2}$ we obtain
\[C_{2}\subset \bar{C}_{2}\subset \acute{C}_{2}\]
 and
therefore
\begin{eqnarray*}
C_{2}& =&C_{2}\cap \mathcal{M}_{X_{1}}(X_{2},t_{2})\\
& \subset &\bar{C}_{2}\cap \mathcal{M}_{X_{1}}(X_{2},t_{2})\\
&\subset&  \acute{C}_{2}\cap \mathcal{M}_{X_{1}}(X_{2},t_{2})\\
&=&C_{2}
\end{eqnarray*}
so that we conclude that
\begin{equation}
\label{eq_gggg}
C_{2}=\bar{C}_{2}\cap \mathcal{M}_{X_{1}}(X_{2},t_{2})\, .
\end{equation}

It is easy to show that both $\mathcal{M}_{X_{1}}(X,t) \subset \mathcal{M}(X,t)$ and
$ \mathcal{M}_{X_{1}}(X_{2},t_{2}) \subset  \mathcal{M}(X_{2},t_{2})$ are extreme subsets.
Therefore, it follows from Lemma \ref{lem_affineedtreme2}
that
\begin{equation}
\label{C0}
 \ext(C_{0})=\ext(C)\cap \mathcal{M}_{X_{1}}(X,t)\,
 \end{equation}
and
\begin{equation}
\label{C2}
\ext(C_{2})=\ext(\bar{C}_{2})\cap \mathcal{M}_{X_{1}}(X_{2},t_{2})\, .
\end{equation}
Since $\iota$  is a composition of affine bijections, it is an affine bijection, so  that we have
\[\ext(C_{2})=\iota  \ext(C_{0})\, .\]

Finally,  observe that
 $\mu^{*}$, selected at the beginning of the proof, satisfies $\mu{*} \in \mathcal{M}_{X_{1}}(X,t)$. Therefore  it
follows that $C_{0}$ and therefore  $C_{2}:=\iota C_{0}$ and $\bar{C}_{2}$
 are not empty.
 Consequently, since $\bar{C}_{2}\subset
\mathcal{M}(X_{2},t_{2})$ is closed and $\mathcal{M}(X_{2},t_{2})$ compact  it follows that $\bar{C}_{2}$ is compact, and since
$\mathcal{M}(X_{2},t_{2})$ is locally convex and metrizable, it follows from Choquet's Theorem  for metrizable compact convex sets, see  Alfsen \cite[Cor.~I.4.9]{Alfsen}, that each element $\mu \in \bar{C}_{2}$ has an integral representation
over the boundary $\ext(\bar{C}_{2})$. That is, $ \ext(\bar{C}_{2}) \neq \emptyset$ is  measurable, and  for $\mu \in \bar{C}_{2}$ there exists a probability measure $p$
on  $\ext(\bar{C}_{2})$   such that for all continuous functions $f$ on $\bar{C}_{2}$, we have
\[\mu(f)=\int_{\ext(\bar{C}_{2})}{\nu(f)dp(\nu)}\, .\]
where $\mu(f)$ and $\nu(f)$ denote the integrals $\int{fd\mu}$ and $\int{fd\nu}$.

Consider the open subset $X_{1} \subset X_{2}$. Since $X_{1}$ is a metric space, it follows,
see e.g.~\cite[Cor.~3.14]{AliprantisBorder:2006}, that the indicator function
$\eins_{X_{1}}$ is the increasing pointwise limit of a sequence of continuous functions $f_{n}, n \in \mathbb{N}$
 with values in
$[0,1]$. Since $\bar{C}_{2}$ is a subset of a metrizable second countable space, it too is
 metrizable and  second countable,  and therefore it follows from \cite[Lem.~3.4]{AliprantisBorder:2006}
that it is separable. Consequently,
 \cite[Thm.~15.13]{AliprantisBorder:2006}  implies that the function $\nu \mapsto \nu(f)$
 is measurable for all  bounded measurable  functions $f$. Therefore,
  by the monotone convergence theorem \cite[Thm.~1.6.2]{Ash:1972} applied  three times:
 to the left hand side,  to the integrand of the righthand side, and
 to the integral on the righthand side, we conclude
that
\begin{equation}
\label{eq_uuu}
\mu(X_{1})=\int_{\ext(\bar{C}_{2})}{\nu(X_{1})dp(\nu)}\, .
\end{equation}
Since $C_{2} \subset \bar{C}_{2}$, it follows that  $\mu \in C_{2}$ has a representing measure $p$
such that    integral formula \eqref{eq_uuu} holds.
Since $\mu \in C_{2}$, the equality
$\mu(X_{1})=1$  implies that $\nu(X_{1})=1$ $p$-almost everywhere. In particular, there exists a $\nu \in  \bar{C}_{2}$
such that $\nu(X_{1})=1$. That is, $\ext(\bar{C}_{2})\cap \mathcal{M}_{X_{1}}(X_{2},t_{2})\neq \emptyset$. Since
by \eqref{C2}  $\ext(C_{2})=\ext(\bar{C}_{2})\cap \mathcal{M}_{X_{1}}(X_{2},t_{2})$ it follows
that $\ext(C_{2}) \neq \emptyset$. Furthermore,
 the relation $ \iota \ext(C_{0})=\ext(C_{2})$ implies that
$\ext(C_{0})\neq \emptyset$,  and the relation $\ext(C_{0})=\ext(C)\cap \mathcal{M}_{X_{1}}(X,t)$
implies that $\ext(C)\neq \emptyset$,
 which is the assertion of the theorem.

\subsection{Proof of Theorem \ref{thm_wasser}}
It is straightforward to show that $X\times X$ is a Borel subset of the Polish metric space determined by
the product of the ambient Polish metric spaces. Therefore,
 Suslin's Theorem, see e.g.~Kechris \cite[Thm.~14.2]{Kechris:1995}, implies that  both $X$ and $X\times X$ are
Suslin, and therefore by
Dellacherie and Meyer  \cite[III.69]{DellacherieMeyer:1975}, it follows 
that all probability measures in both $\mathcal{M}(X)$ and $\mathcal{M}(X\times X)$ are tight. This tightness facilitates both the existence of extreme points for  convex sets
 of measures, useful  in obtaining the assertion,   and
the duality theorems of Strassen and Kantorovich-Rubinstein used in the proof of Theorem \ref{thm_wozabal}.

Lemma \ref{lem_lsc} implies that $\{\nu \in  \mathcal{M}(X\times X): \int{c(x,x')d\nu(x,x')}\leq \epsilon\}$
is closed and convex in the weak topology. Moreover,  by Aliprantis and Border \cite[Thm.~15.14]{AliprantisBorder:2006}
the marginal maps $P_{1}$ and $P_{2}$ are continuous in the weak topologies. Since singletons in $\mathcal{M}(X)$ are closed, for $\mu \in \mathcal{M}(X)$,
it follows that $\{\nu \in  \mathcal{M}(X\times X):P_{1}\nu=\mu_{n}\}$,
  $\{\nu \in  \mathcal{M}(X\times X):P_{2}\nu=\mu\}$ are also closed and convex,  and therefore
$\Gamma_{\mu_{n},\epsilon} \cap P_{2}^{-1}\mu$ is closed and convex in the weak topology.
Since $\Gamma_{\mu_{n},\epsilon} \cap P_{2}^{-1}\mu$ is nonemtpy,   Winkler's Theorem \ref{thm_winkler}
implies
  that it possesses an extreme point.
Therefore
Lemma \ref{lem_affineedtreme2}
 implies that
\[ P_{2}(\ext(\Gamma_{\mu_{n},\epsilon})) \supset \ext\bigl(P_{2}(\Gamma_{\mu_{n},\epsilon})\bigr)\, , \]
establishing the second assertion.

 For the first,  let us describe $\ext(\Gamma_{\mu_{n},\epsilon})$. To that end,
 write $\mu_{n}=\sum_{i=1}^{n}{\alpha_{i}\d_{x_{i}}}$
with $\alpha_{i}\geq 0, x_{i}\in X, i=1,..,n$ and $\sum_{i=1}^{n}{\alpha_{i}}=1$.
Then
 consider the $n+1$ constraint functions $c$ and
$\eins_{\{x_{i}\}\times X}, i=1,..,n$ to define $\Gamma_{\mu_{n},\epsilon}$ as  inequality/equality constraints defined by integrals of measurable functions on
$\mathcal{M}(X\times X)$.  Then
\cite[Thm.~4.1, Rmk.~4.2]{OSSMO:2011} (derived from Winkler \cite[Thm.~2.1]{Winkler:1988}, which is a
consequence of Dubins \cite{Dubins}) implies that
\[\ext(\Gamma_{\mu_{n},\epsilon}) \subset \Delta_{n+2}(X\times X)\, ,\]
establishing the first
 assertion. The third assertion follows by combining the first two and
$P_{2}\bigl(\Delta_{n+2}(X\times X)\bigr)=\Delta_{n+2}(X)$.

\subsection{Proof of Theorem \ref{thm_wozabal}}
Since $X$ is a Borel subset in a Polish metric space,
 Suslin's Theorem, see e.g.~Kechris \cite[Thm.~14.2]{Kechris:1995}, implies  that  $X$  is
Suslin, and therefore by
Dellacherie and Meyer  \cite[III.69]{DellacherieMeyer:1975}, it follows that
 all probability measures in $\mathcal{M}(X)$ are tight.

Let us first begin with the Prokhorov case.
 We use the Prokhorov metric on $\mathcal{M}(X\times X)$.
Consider the subset $\Gamma_{\mu_{n},\epsilon}\subset \mathcal{M}(X\times X)$ defined by
\[ \Gamma_{\mu_{n},\epsilon}:= \Bigl\{\nu \in \mathcal{M}(X\times X):  \nu\{d>\epsilon\}\leq \epsilon,
\, P_{1}\nu=\mu_{n}\Bigr\}\, .\]
For any $\nu \in \Gamma_{\mu_{n},\epsilon}$, for $\mu':=P_{2}\nu$ it follows that
$P_{1}\nu=\mu_{n}$, $P_{2}\nu=\mu'$ and
$\nu\{d>\epsilon\}\leq \epsilon$, so that
by the Prokhorov-Ky Fan inequality \cite[Thm.~11.3.5]{Dudley:2002} it follows that
$d_{Pr}(\mu',\mu_{n}) \leq \epsilon$, that is $\mu' \in B_{\epsilon}(\mu_{n})$, so that we conclude
that
\begin{equation}
\label{f1} P_{2}(\Gamma_{\mu_{n},\epsilon}) \subset B_{\epsilon}(\mu_{n})\, .
\end{equation}

To obtain the reverse inequality,
  let us  first note that the $\inf$ in the definition
\eqref{def_Pr} of the Prokhorov metric can be replaced by a $\min$. To see this,
observe that for fixed $A \in \mathcal{B}(X)$, that the parametrized family of open sets
$A^{\epsilon}, \epsilon >0$ is increasing. Consequently, if $\epsilon_{n}\downarrow \epsilon'$,
then for any $\mu \in \mathcal{M}(X)$ we have
$\mu(A^{\epsilon_{n}}) \downarrow \mu(A^{\epsilon'}) $, so that,  for fixed $A \in \mathcal{B}(X)$  and
$\mu_{1},\mu_{2} \in \mathcal{M}(X)$, the interval
$\{\epsilon:\mu_{1}(A) \leq \mu_{2}(A^{\epsilon})+\epsilon\}$
is closed. It follows that the intersection of these closed intervals
$\{\epsilon:\mu_{1}(A) \leq \mu_{2}(A^{\epsilon})+\epsilon, \, A \in  \mathcal{B}(X)\}$
 over all $A \in  \mathcal{B}(X)$ is closed. Therefore the infimum in
the definition \eqref{def_Pr} is attained.

Now consider  $\mu \in B_{\epsilon}(\mu_{n})$ and define $\epsilon^{*}:=d_{Pr}(\mu_{n},\mu)$.
Then by the previous remark we have
\[ \mu(A) \leq \mu_{n}(A^{\epsilon^{*}})+\epsilon^{*},\quad  A \in \mathcal{B}(X)\,\, \]
and the inequality $\epsilon^{*}\leq \epsilon$ implies that
\[ \mu(A) \leq \mu_{n}(A^{\epsilon})+\epsilon,\quad  A \in \mathcal{B}(X)\, . \]
 Moreover,  if we denote
$d(x,A):=\inf_{y \in A}{d(x,y)}$ then it is easy to see that
$A^{\epsilon}=\{x\in X: d(x,A) < \epsilon\}$ and defining
$A^{\epsilon \small ]}=\{x\in X: d(x,A) \leq \epsilon\}$  we obtain that
\[ \mu(A) \leq \mu_{n}(A^{\epsilon]})+\epsilon,\quad  A \in \mathcal{B}(X)\, . \]
Then, since both  $\mu$ and $\mu_{n}$ are tight, Dudley's \cite[Thm.~11.6.2]{Dudley:2002}
extension of Strassen's Theorem to tight measures on separable metric spaces
 implies that there exists a probability measure $
\nu \in \mathcal{M}(X\times X)$ such that $P_{1}\nu=\mu_{n}$, $P_{2}\nu=\mu$
and
$\nu\{
d>\epsilon\}  \leq \epsilon$,
that is, there exists a $\nu \in \Gamma_{\mu_{n},\epsilon}$ such that $P_{2}\nu =\mu$, so that
we obtain
\[P_{2}\bigl(\Gamma_{\mu_{n},\epsilon}\bigr)\supset B_{\epsilon}(\mu_{n}) \]
and,  so by \eqref{f1}, conclude that
\begin{equation}
\label{strassen}  P_{2}\bigl(\Gamma_{\mu_{n},\epsilon}\bigr)=B_{\epsilon}(\mu_{n})\, .
\end{equation}

Since the metric $d$ is a continuous function, it follows that the set $\{
(x,x') \in X\times X: d(x,x')>\epsilon\}$ is open and therefore the indicator
function $\eins_{d>\epsilon}$
 is lower semicontinuous. Therefore,
we can apply Theorem \ref{thm_wasser} to obtain
\begin{eqnarray*}
\ext\bigl(B_{\epsilon}(\mu_{n})\bigr) &=& \ext\bigl(P_{2}(\Gamma_{\mu_{n},\epsilon})\bigr)\\
&\subset& \Delta_{n+2}(X)
\end{eqnarray*}
establishing the assertion.

Now let us consider the
 Kantorovich case. To that end,
 let $\mathcal{M}_{1}(X) \subset \mathcal{M}(X)$  denote those Borel probability measures  $\mu$ such that
$\int{d(x',x)d\mu(x)}<\infty$ for some $x' \in X$, and
consider  the Monge-Wasserstein distance $d_{W}$ on $ \mathcal{M}_{1}(X)$
defined by
\[d_{W}(\mu_{1},\mu_{2}):=\inf_{\nu \in M(\mu_{1},\mu_{2})}{\int{d(x,x')d\nu(x,x')}}\, .\]
Then the Kantorovich-Rubinstein Theorem \cite[Thm.~11.8.2]{Dudley:2002} states that for all
$\mu_{1},\mu_{2} \in \mathcal{M}_{1}(X)$ we have
\[d_{K}(\mu_{1},\mu_{2})=d_{W}(\mu_{1},\mu_{2})\, ,\]
and if $\mu_{1}$ and $\mu_{2}$ are tight, that there is a measure
in $\mathcal{M}(X\times X)$ at which  the infimum in the definition of $d_{W}$ is attained.

Define $\Gamma_{\mu_{n},\epsilon} \subset \mathcal{M}(X\times X)$ by
\[ \Gamma_{\mu_{n},\epsilon}:= \Bigl\{\nu \in \mathcal{M}(X\times X): \int{d(x,x')d\nu(x,x')} \leq \epsilon,
\, P_{1}\nu=\mu_{n}\Bigr\} \, , \] and for
 $\nu \in \Gamma_{\mu_{n},\epsilon}$, consider $\mu:=P_{2}\nu$. Then, for $y \in X$, we have
\begin{eqnarray*}
\int{d(y,x')d\mu(x')}&=& \int{d(y,x')d\nu(x,x')}\\
&\leq& \int{\bigl(d(y,x)+d(x,x')\bigr)d\nu(x,x')}\\
&=& \int{d(y,x)d\nu(x,x')}+\int{d(x,x')d\nu(x,x')}\\
&=& \int{d(y,x)d\mu_{n}(x)}+\int{d(x,x')d\nu(x,x')}\\
&\leq& \int{d(y,x)d\mu_{n}(x)}+ \epsilon\, ,
\end{eqnarray*}
and since $\mu_{n}$ is a finite convex  sum of Dirac masses, it follows that
$\int{d(y,x')d\mu(x')} <\infty$, that is,  $P_{2}\nu \in \mathcal{M}_{1}(X)$, so that we conclude that
\[ P_{2}(\Gamma_{\mu_{n},\epsilon}) \subset \mathcal{M}_{1}(X)\, .\]

Since
  all measures
 in $\mathcal{M}_{1}(X)$ are tight,
   the Kantorovich-Rubinstein Theorem
then implies that
\[ P_{2}(\Gamma_{\mu_{n},\epsilon})=B_{\epsilon}(\mu_{n})\, \]
in the same way that the  Strassen Theorem implied it in \eqref{strassen} for the Prokhorov metric.
Moreover, since $d$ is a metric, it  is non-negative, real-valued and continuous,
 so  it follows that it is a non-negative semicontinuous
real-valued function.  As in the Prokhorov case, Theorem \ref{thm_wasser}  then yields the assertion.

\subsection{Proof of Lemma \ref{lem_eee}}
Since an element $\nu \in \Delta_{n+2}(X\times X) $  may have support smaller than $n+2$,  we
represent it  by
$\nu=\sum_{i=1}^{m}{\alpha_{i}\d_{x_{i},x'_{i}}}$, $\alpha_{i}> 0,x_{i},x'_{i}\in X, i=1,..,m, \sum_{i=1}^{m}{\alpha_{i}}=1$, for $m \leq n+2$, where we also require $(x_{i},x'_{i})\neq (x_{j},x'_{j}), i\neq j$.
Such an element $\nu \in \Delta_{n+2}(X\times X) $ is a member of
$P_{1}^{-1}\mu_{n}\cap \Delta_{n+2}(X\times X)$  if and only if
$P_{1}\nu=\mu_{n}$. Therefore, we conclude that $\nu \in P_{1}^{-1}\mu_{n}\cap \Delta_{n+2}(X\times X)$
if and only if
\[\sum_{j=1}^{m}{\alpha_{j}\d_{x_{j}}} = \sum_{i=1}^{n}{\beta_{i}\d_{y_{i}}}\, .
\]
Since $\beta_{i}>0,i=1,..,n$ and $\alpha_{j}>0,j=1,..,m$ it follows that
\[\{x_{j},j=1,..,m\} = \{y_{i},i=1,..,n\}\, .\]

In particular,
 $m$ must satisfy $n \leq m \leq n+2$. Moreover,  the three possible cases
 $m=n, n+1, n+2$ appear as follows: when $m=n$,
 there is a relabeling of the indices of $(x_{j}, x'_{j}), j=1,..,n$  so that
$x_{i}=y_{i}, \alpha_{i}=\beta_{i}, i=1,..,n$.  When $m=n+1$, there is a $j_{1}\in \{1,..,n\}$ and a relabeling so that
$x_{i}=y_{i}, i=1,..,n$ and $x_{n+1}=y_{j_{1}}$.  Then we also have
$\alpha_{i}=\beta_{i},i \neq j_{1}$ and $\alpha_{j_{1}}+\alpha_{n+1}=\beta_{j_{1}}$.
When $m=n+2$, then there is a relabeling so that
$x_{i}=y_{i}, i=1,..,n$ and either  1) there is a  $j_{1}\in \{1,..,n\}$ such that
 $x_{n+1}=x_{x+2}=y_{j_{1}}$ and
$\alpha_{i}=\beta_{i},i \neq j_{1}$ and $\alpha_{j_{1}}+\alpha_{n+1}+\alpha_{n+2}=\beta_{j_{1}}$ or 2)
there are two distinct values $j_{1}, j_{2} \in \{1,..,n\}$  such
that   $x_{n+1}=y_{j_{1}}$,  $x_{n+2}=y_{j_{2}}$,
$\alpha_{i}=\beta_{i},i \neq j_{1}\,i \neq j_{2}$,  $\alpha_{j_{1}}+\alpha_{n+1}=\beta_{j_{1}}$ , and
$\alpha_{j_{2}}+\alpha_{n+2}=\beta_{j_{2}}$.
It is clear the the $m=n$ case amounts to the statement $\nu \in \Pi_{0}$ defined in \eqref{def_pi0}. Let us now
show that the $m=n+1$ and $m=n+2$ cases amount to
the statements $\nu \in \Pi_{i}$ for some $i$
and $\nu \in \Pi_{i,j}$ for some $i\leq j$, defined in  \eqref{def_p1}, and \eqref{def_p2} respectively, establishing
the assertion.

To that end, for the $m=n+1$ case, the above assertion states that there is an $i \in \{1,..,n\}$
and an $x \in X^{n+1}$ such that
\[\nu =\sum_{k\neq i, k \in \{1,n\}}{\beta_{k} \d_{y_{k},x_{k}}} +\alpha_{i}\d_{y_{i},x_{i}}+\alpha_{n+1}
\d_{y_{i},x_{n+1}}\]
with $\alpha_{i}+\alpha_{n+1}=\beta_{i}$.
Since
\begin{eqnarray*}
\sum_{k\neq i, k \in \{1,n\}}{\beta_{k} \d_{y_{k},x_{k}}} +\alpha_{i}\d_{y_{i},x_{i}}+\alpha_{n+1}\d_{y_{i},x_{n+1}}
&=& \delta_{y,x} +(\alpha_{i}-\beta_{i})\d_{y_{i},x_{i}}+\alpha_{n+1}\d_{y_{i},x_{n+1}}\\
&=& \delta_{y,x} +\alpha_{n+1}(\d_{y_{i},x_{n+1}}-
\d_{y_{i},x_{i}})\, ,
\end{eqnarray*}
by the identification $\gamma:=\alpha_{n+1}$, we conclude that $\nu \in \Pi_{i}$ defined in \eqref{def_p1}. The proof in the $m=n+2$ case is essentially the same.

\subsection{Proof of Lemma \ref{lem_eee2}}
 Let us define
\begin{eqnarray}
\label{def_theta}
\Theta&:=&   \Bigl\{\nu \in P^{-1}_{1}\mu_{n}\cap \Delta_{n+2}(X\times X) :\nu=\sum_{i=1}^{m}{\alpha_{i}\d_{x_{i},x'_{i}}}, 1 \leq m \leq n+2,
\alpha_{i} >0, x_{i},x'_{i} \in X, i=1,..,m,\notag \\
&& \text{the vectors}\,\, \bigl(\eins_{y_{1}}(x_{i}),..., \eins_{y_{n}}(x_{i}), \eins_{d(x_{i},x'_{i})>\epsilon},1\bigr), \, i=1,..,m\,\,
\text{are linearly independent}
 \Bigr\}\, . \notag\\
\end{eqnarray}
 Then
 the identity
\[ \Gamma_{\mu_{n},\epsilon}=P^{-1}_{1}\mu_{n} \cap \bigl\{\nu \in \mathcal{M}(X\times X): \nu\{d> \epsilon\} \leq \epsilon \bigr\}\] implies that
\begin{equation}
\label{eq_ttt}
\bar{\Theta}=\Theta \cap \bigl\{\nu \in \mathcal{M}(X\times X): \nu\{d> \epsilon\} \leq \epsilon \bigr\}\, .
\end{equation}

As in Section \ref{sec_comp_p1},
let us compute $\bar{\Theta}$ by first computing $\Theta$ and then using the
identity \eqref{eq_ttt}.
To that end,  observe that the definition \eqref{def_theta}  of $\Theta$ implies that the support points
$(x_{i},x'_{i}), i=1,..,m$ contain no duplicates so that we can apply Lemma \ref{lem_eee} which implies that
  we can constrain the values of $m$ in the definition of $\Theta$ to
 $n \leq m \leq n+2$. Moreover, $\Theta$ is defined in terms
of $P^{-1}_{1}\mu_{n}\cap \Delta_{n+2}(X\times X)$, and by Lemma \ref{lem_eee} we have
$P^{-1}_{1}\mu_{n}\cap \Delta_{n+2}(X\times X)=\Pi_{0}  \cup_{k=1}^{n}\Pi_{k} \cup_{i\leq j} \Pi_{i,j}$.
Consequently, using the multiindex $\imath$ introduced above  \eqref{def_barpi0}, it is natural to define
\begin{eqnarray*}
\Theta_{\imath}&:=&\Theta \cap \Pi_{\imath}
\end{eqnarray*}
and
observe that
\[\Theta=\Theta_{0}  \cup_{k=1}^{n}\Theta_{k} \cup_{i\leq j} \Theta_{i,j}.\]

First consider $\Theta_{0}$. Since
the definition of $\Pi_{0}$ implies that $\{x_{j},j=1,..,n\}$ must be a permutation
of $\{y_{i},i=1,..,n\}$, it follows that the linear independence condition of \eqref{def_theta} is satisfied in this case.
That is,
\begin{equation}
\label{theta0}
\Theta_{0}=\Pi_{0}\, .
\end{equation}
Now consider $\Pi_{i}$ for $i \in \{1,..,n\}$.
Then the definition \eqref{def_p1} of $\Pi_{i}$ implies that, upon relabeling, that
the linear independence
of the set
$\bigl(\eins_{y_{1}}(x_{i}),..., \eins_{y_{n}}(x_{i}), \eins_{d(x_{i},x'_{i})>\epsilon},1\bigr), i=1,..,n+1$ amounts to
the linear independence of the set
\begin{eqnarray*}
  \Bigl(I_{n\times n}, &z_{n}&,I_{n}\Bigr)
\end{eqnarray*}
together with
\begin{eqnarray*}
\Bigl(0,.., 1_{i},..,0 , &\eins_{d(y_{i},x'_{n+1})>\epsilon}&,1\Bigr)
\end{eqnarray*}
where $z_{n}$ has components $\eins_{d(y_{i},x'_{i})>\epsilon}, i=1,..,n$, $I_{n\times n}$ is the identity matrix,
$I_{n}$ is the vector of $1$s,
and $1_{i}$  indicates a $1$ in the $i$-th position.
Because the first row has the identity matrix, this set of vectors is linearly independent if and only if
\begin{eqnarray*}
\Bigl(0,.., 1_{i},..,0 , &\eins_{d(y_{i},x'_{i})>\epsilon}&,1\Bigr)\\
\Bigl(0,.., 1_{i},..,0 , &\eins_{d(y_{i},x'_{n+1})>\epsilon}&,1\Bigr)
\end{eqnarray*}
is linearly independent, which
is equivalent to the assertion that  $x'\in \Lambda_{i}$ defined in \eqref{lambda1}. Consequently, we obtain
\begin{equation}
\label{theta1}
 \Theta_{i} =\Pi_{i} \cap \Lambda_{i} \, .
\end{equation}
For $\Theta_{i,j}$ with $i \leq j$,  let us first show that  $\Theta_{i,i}=\emptyset$.
 To that end, let $x'\in X^{n+2}$ and consider
$\nu \in \Pi_{i,i}(x')$. Then using the same reasoning as above,  it follows that the linear independence condition
is equivalent to the linear independence of the three vectors
\begin{eqnarray*}
  \Bigl(0,.., 1_{i},..,0 , &\eins_{d(y_{i},x'_{i})>\epsilon}&,1\Bigr)\\
\Bigl(0,.., 1_{i},..,0 , &\eins_{d(y_{i},x'_{n+1})>\epsilon}&,1\Bigr)\\
\Bigl(0,.., 1_{i},..,0 , &\eins_{d(y_{i},x'_{n+2})>\epsilon}&,1\Bigr)\, .
\end{eqnarray*}
Since the  last row is identically $1$, the independence of this set is not possible regardless of the values
of $\eins_{d(y_{i},x'_{i})>\epsilon},\eins_{d(y_{i},x'_{n+1})>\epsilon}$ and
$\eins_{d(y_{i},x'_{n+2})>\epsilon}$. Therefore,
\begin{equation}
\label{theta2a}
\Theta_{i,i}=\emptyset, \quad i=1,..,n\, .
\end{equation}
So let us consider $\Theta_{i,j}$ with $i < j$.  Then, upon relabeling, the linear independence
of the set
$\bigl(\eins_{y_{1}}(x_{i}),..., \eins_{y_{n}}(x_{i}), \eins_{d(x_{i},x'_{i})>\epsilon},1\bigr), i=1,..,n+2$ amounts to
the linear independence of the set
\begin{eqnarray*}
  \Bigl(I_{n\times n}, &z_{n}&,I_{n}\Bigr)
\end{eqnarray*}
together with
\begin{eqnarray*}
\Bigl(0,.., 1_{i},..,0,..,0 , &\eins_{d(y_{i},x'_{n+1})>\epsilon}&,1\Bigr)\\
\Bigl(0,..,0,.., 1_{j},..,0 , &\eins_{d(y_{j},x'_{n+2})>\epsilon}&,1\Bigr)\, .
\end{eqnarray*}
Because the first row has the identity matrix, this set of vectors is linearly independent if and only if
both
\begin{eqnarray*}
\Bigl(0,.., 1_{i},..,0 , &\eins_{d(y_{i},x'_{i})>\epsilon}&,1\Bigr)\\
\Bigl(0,.., 1_{i},..,0 , &\eins_{d(y_{i},x'_{n+1})>\epsilon}&,1\Bigr)
\end{eqnarray*}
and
\begin{eqnarray*}
\Bigl(0,.., 1_{j},..,0 , &\eins_{d(y_{j},x'_{j})>\epsilon}&,1\Bigr)\\
\Bigl(0,.., 1_{j},..,0 , &\eins_{d(y_{j},x'_{n+2})>\epsilon}&,1\Bigr)
\end{eqnarray*}
are linearly independent.
 Then, as in the
 $\Theta_{i}$ case above, the  linear independence of these two sets is equivalent to  requiring
that  $x'\in \Lambda_{i,j}$ defined in \eqref{lambda2}.
That is, we have
\begin{equation}
\label{theta2}
\Theta_{i,j}=\Pi_{i,j}\cap  \Lambda_{i,j}\,.
\end{equation}
Therefore, we have established that
\[\Theta= \Pi_{0}  \cup_{k=1}^{n}(\Pi_{i}\cap  \Lambda_{i}) \cup_{i< j}(\Pi_{i,j}\cap  \Lambda_{i,j}) \, , \]
and  the assertion then easily follows.

\section*{Acknowledgments}
The authors thank  the referees for a thorough and thoughtful review of the manuscript providing many substantial improvements in its presentation.
The authors gratefully acknowledge this work supported by the Air Force Office of Scientific Research and the DARPA EQUiPS Program under Award Number FA9550-12-1-0389 (Scientific Computation of Optimal Statistical Estimators) and number FA9550-16-1-0054 (Computational Information Games)
\bibliographystyle{plain}
\bibliography{./refs}
\newpage
\newpage
\newpage
          %







          %


          %


          %


          \end{document}